\documentclass[a4paper,oneside,11pt]{article}

\usepackage{amsmath,amsfonts,amscd,amssymb,amsbsy}
\usepackage{longtable,geometry}
\usepackage[english]{babel}
\usepackage[utf8]{inputenc}
\usepackage[active]{srcltx}
\usepackage[T1]{fontenc}
\usepackage{graphicx}
\usepackage{pstricks}
\usepackage{bbm}
\usepackage{mathtools}

\usepackage{MnSymbol}
\usepackage{stmaryrd}
\usepackage{nicefrac}
\usepackage{calrsfs}
\usepackage{enumitem}

\usepackage{tocloft}

\usepackage{rotating}

\usepackage{xcolor}
\usepackage{framed}
\usepackage{hyperref}
\hypersetup{
    colorlinks=false,
    pdfborder={0 0 0},
}

\usepackage{cleveref}

\colorlet{shadecolor}{blue!15}

\newtheorem{theorem}{Theorem}[section]

\newtheorem{lemma}[theorem]{Lemma}
\newtheorem{proposition}[theorem]{Proposition}
\newtheorem{definition}[theorem]{Definition}

\newtheorem{remark}[theorem]{Remark}
\newtheorem{claim}[theorem]{Claim}

\newcommand{\be}[1]{\begin{equation}\label{#1}}
\newcommand{\ee}{\end{equation}}
\numberwithin{equation}{section}

\newcommand{\ba}[1]{\begin{align}\label{#1}}
\newcommand{\ea}{\end{align}}
\numberwithin{equation}{section}

\newcommand{\ben}{\begin{equation*}}
\newcommand{\een}{\end{equation*}}
\numberwithin{equation}{section}

\newenvironment{proof}[1][\relax]
  {\paragraph{Proof\ifx#1\relax\else~of #1\fi:}}%
  {~\hfill$\square$\par\bigskip}

\newcommand{\calA}{\mathcal{A}}
\newcommand{\calB}{\mathcal{B}}
\newcommand{\calC}{\mathcal{C}}

\newcommand{\calE}{\mathcal{E}}

\newcommand{\calG}{\mathcal{G}}

\newcommand{\eps}{\varepsilon}

\newcommand{\La}{\Lambda}

\setlist[itemize]{itemsep=0pt, topsep=2pt}
\setlist[enumerate]{itemsep=1pt, topsep=4pt}

\newcommand{\rk}[1]{\bgroup\color{red}%
  \par\medskip\hrule\smallskip%
  \noindent\textbf{#1}%
  \par\smallskip\hrule\medskip\egroup}

\title{Near critical scaling relations for planar Bernoulli percolation without differential inequalities}
\author{Hugo Duminil-Copin\footnote{Universit\'e de Gen\`eve (hugo.duminil@unige.ch) and Institut des Hautes \'Etudes Scientifiques (duminil@ihes.fr)}, \ 
Ioan Manolescu\footnote{Universit\'e de Fribourg (ioan.manolescu@unifr.ch)}\ \ and 
Vincent Tassion\footnote{ETHZ (vincent.tassion@math.ethz.ch)}}
\date{\today}

\begin{document}
\maketitle

\begin{abstract}
	We provide a new proof of the near-critical scaling relation $\beta=\xi_1\nu$ for Bernoulli percolation on the square lattice already proved by Kesten in 1987. 
    We rely on a novel approach that does not invoke Russo's formula, but rather relates differences in crossing probabilities at different scales. The argument is shorter and more robust than previous ones and is more likely to be adapted to other models. The same approach may be used to prove the other scaling relations appearing in Kesten's work.
\end{abstract}

\section{Introduction}

Consider the square lattice  with vertex-set $\mathbb Z^2$ and edge-set $E(\mathbb Z^2):=\{\{x,y\}\subset \mathbb Z^2:\|x-y\|_2=1\}$. Let $\mathrm P_p$ be  the Bernoulli percolation measure on $\{0,1\}^{E(\mathbb Z^2)}$ defined as the law of i.i.d.~Bernoulli random variables $(\omega(e):e\in E(\mathbb Z^2))$  with parameter $p$.  A configuration $\omega\in \{0,1\}^{E(\mathbb Z^2)}$ is identified with the graph with vertex-set $\mathbb Z^2$ and edge-set $\{e\in E(\mathbb Z^2):\omega(e)=1\}$, and paths in $\omega$ are identified to continuous, piecewise linear paths drawn in the plane. 

We will assume basic knowledge of Bernoulli percolation and refer to \cite{Gri99} for background. This paper is meant to be largely self-contained; the only tool specific to two-dimensional percolation that is used is the Russo-Seymour-Welsh (RSW) theory, which is discussed in Appendix~\ref{sec:lem}.

For $n\ge0$ and $x\in \mathbb Z^2$, define $\Lambda_n:=[-n,n]^2$ and $\partial\Lambda_n:=\Lambda_n\setminus\Lambda_{n-1}$, as well as their translates $\Lambda_n(x)$ and $\partial\Lambda_n(x)$ by $x$. 
We will often identify the these sets with the sets of vertices or edges contained in them. 
For any $r\ge0$ and $p\in[0,1]$, introduce the quantities
\begin{align*}
	\theta_r(p)&:=\mathrm P_p[0\text{ is connected to }\partial\Lambda_{\lfloor r \rfloor }],\\
	\theta(p)&:=\lim_{N\rightarrow\infty}\theta_N(p),\\
	\xi(p)&:=\lim_{N\rightarrow\infty} - N/\log[\theta_N(p)-\theta(p)],
\end{align*}
where $\theta(p)$ and $\xi(p)$ are respectively called the {\em infinite cluster density} and the {\em correlation length} (the justification that the limits exist is classical: the first one by monotonicity, the second one by a submultiplicativity argument). 

Kesten \cite{Kes80} proved that Bernoulli percolation on the square lattice undergoes a continuous \emph{sharp} phase transition at $p=p_c=1/2$ in the sense that
\begin{itemize}
\item $\theta(p)>0$ for every $p>p_c$ and $\theta(p)=0$ for $p\le p_c$, and that 
\item  $\xi(p)<\infty$ for every $p\ne p_c$, which is to say that $\theta_N(p)$ converges to $\theta(p)$ at an exponential rate for $p\neq p_c$.
\end{itemize}

It is expected that the quantities above satisfy the following asymptotics: 
\begin{align*}
\theta_N(p_c)&=N^{-\xi_1+o(1)}\qquad\text{ as $N\rightarrow\infty$,}\\
\theta(p)&=(p-p_c)^{\beta+o(1)}\ \,\text{ as $p\searrow p_c$,}\\
\xi(p)&=|p-p_c|^{-\nu+o(1)}\ \,\text{ as $p\rightarrow p_c$},
\end{align*}
for certain {\em critical exponents} $\xi_1$, $\beta$, and $\nu$, where $o(1)$ denotes a quantity tending to zero.
While these exponents were only proved to exist in the case of site percolation on the triangular lattice~\cite{SW01},
and their values were shown to be $5/48$, $5/36$, and $3/4$, respectively, they are expected to be universal among independent planar percolation models. 
More importantly for us, they are expected to satisfy
\begin{align}
\beta&=\nu\xi_1.\label{eq:aa1}
\end{align}

The above is one of several {\em scaling relations} which relate different critical exponents, and which are expected to hold for more general statistical mechanics models.
For Bernoulli percolation, \eqref{eq:aa1} along with several other scaling relations -- or rather formulas that imply the relations if the critical exponents exist -- were proved in a celebrated paper of Kesten \cite{Kes87}.
These results were recently extended to the more general FK-percolation model \cite{DM20}. 
 
Our main goal is to provide a new proof of Kesten's result; we will focus on \eqref{eq:aa1}.

\begin{theorem}[Scaling relation between $\beta$, $\nu$, and $\xi_1$.]\label{thm:1}
	There exist $c,C>0$ such that for every $p>p_c$,
	\begin{align*}
		c\, \theta_{\xi(p)}(p_c)~&\le~\theta(p)~\le~ C\,\theta_{\xi(p)}(p_c).
	\end{align*}
\end{theorem}

The advantage of our approach, beside the obvious shortness of the exposition, 
is that it does not rely on interpretations (using Russo's formula) of the derivatives of probabilities of increasing events 
using so-called pivotal edges.
Indeed, such formulas are unavailable in most  dependent percolation models, which has long been a limitation to the extension of \cite{Kes87}.
As such, our approach is more robust, and is likely to extend to other models. 
For instance, the authors believe that similar arguments may provide shorter proofs of some (but not all) of the results of \cite{DM20}. 

The main difficulty in deriving Theorem~\ref{thm:1} is to obtain the stability of the one arm event probability below the correlation length.
This is the object of the next theorem.

\begin{theorem}[Stability below the correlation length]\label{thm:stability}
    There exist $C>0$ such that for every $p\ge p_c$ and every $N\le \xi(p)$,
    \begin{align}
    	\theta_N(p_c)\le \theta_N(p)\le C\,\theta_N(p_c)\label{eq:stability pi1}.
    \end{align}
\end{theorem}

Our technique also enables one to derive the other scaling relations of \cite{Kes87}, as the key difficulty is always a stability result similar to Theorem~\ref{thm:stability}; see also Remark~\ref{rem:stability4arm} for additional details.
We choose to focus on \eqref{eq:aa1} to simplify the exposition.
We refer to the original paper and to the review \cite{Nol08} for details and encourage the reader --  as an interesting exercise -- to prove the other relations using the techniques developed here. 
Finally, a full description of the near-critical phase was provided in \cite{GPS}. 
In \cite[Lemma~8.4]{GPS}, the authors propose an alternative approach to proving stability, 
based on the poissonian way in which edges open when the percolation parameter increases. 
Note nevertheless that this approach only proves stability under a scale explicitly computed using the density of pivotals at criticality; which is smaller, but not apriori equal to the correlation length. 

\paragraph{Organization of the paper} In the next section we present the proof of Theorem~\ref{thm:stability} subject to two important propositions that we derive in the following two sections. Theorem~\ref{thm:1} is derived from Theorem~\ref{thm:stability} in Section~\ref{sec:thm1}. 
In order to make this paper essentially self-contained, the appendix contains the proofs of three required (and standard) consequences of the classical Russo-Seymour-Welsh (RSW) inequality \eqref{eq:RSW}. 

While most of our proofs are streamlined, our essential innovation lies in Proposition~\ref{prop:comparison scale Delta}, 
specifically in the definition and use of pivotal switches in Section~\ref{sec:pivotal_switches}.

\paragraph{Some general notation}
We use the canonical increasing coupling  between the percolation measures for different parameters $p$. Let $\mathbf P$ be the product  measure on $[0,1]^{E(\mathbb Z^2)}$ defined as the law of i.i.d uniform random variables on $[0,1]$. For every $p\in[0,1]$, define the  canonical map $\omega_p$ from $ [0,1]^{E(\mathbb Z^2)}$ to $\{0,1\}^{E(\mathbb Z^2)}$ by  setting $\omega_p(e)=\omega_p(e,U):=\mathbf 1_{U(e)\le p}$ for every edge $e$. This way, under $\mathbf P$, each  configuration  $\omega_p$ has law $\mathrm P_p$ and the configurations are naturally ordered almost surely.

Introduce the dual square lattice with vertex-set $(\tfrac12,\tfrac12)+\mathbb Z^2$ and edges between nearest neighbours. For an edge $e$ of the square lattice, let $e^\star$ be the associated dual edge intersecting it in its middle. On the dual square lattice, define the dual configuration $\omega^\star$ associated to $\omega$ by the formula $\omega^\star(e^\star)=1-\omega(e)$. As for $\omega$, $\omega^\star$ is understood as a subgraph of the dual square lattice, and paths in $\omega^\star$ are identified with continuous, piecewise linear paths drawn in the plane.
Observe that if $\omega$ has law $\mathrm P_p$, then $\omega^\star(e^\star)$ are i.i.d.~Bernoulli random variables of parameter $1-p$. 

For a rectangle $R:=[a,b]\times[c,d]$, let $\mathcal C(R)$ denote the event that $R$ is crossed by a path in $\omega$ from left to right. 

\paragraph{Acknowledgments} 
This project has received funding from the European Research Council (ERC) under the European Union's Horizon 2020 research and innovation programme (grant agreements No.~757296 and No.~851565). The authors acknowledge funding from the NCCR SwissMap and the Swiss FNS.

\section{Proof of Theorem~\ref{thm:stability}}

The proof will rely on two propositions involving the quantity $\Delta_n(p,p')$ that we now introduce. 
For $n\ge1$ and $p'\ge p\ge p_c$, set
\begin{align}\label{eq:4}
	\Delta_n(p,p')&:=\mathbf P[\omega_{p'}\in \mathcal C(H_n),\,\omega_p\notin \mathcal C(V_n)],
\end{align}
where $H_n:=[-2n,2n]\times[-n,n]$ and $V_n:=[-n,n]\times[-2n,2n]$.
Notice that the event $\calC(H_n)$ is contained in $\calC(V_n)$ and thus $\Delta_n(p,p) = 0$.
For $p < p'$, $\Delta_n(p,p')$ will be shown to encode how crossing probabilities at scale $n$ increase when the parameter increases from $p$ to $p'$. 

Write $\calE_n$ for the set of pairs of configurations $(\omega,\omega')$ with $\omega \le \omega'$ 
and such that there exists $e\in \Lambda_n$ 
and two {\em disjoint} paths $\gamma$ in $\omega'$ and $\gamma^\star$ in $\omega^\star$ 
that run from $\partial\Lambda_{2n}$ to itself and intersect only at the center of $e$. 
Call an edge as above a {\em mixed pivotal} at scale~$n$. 

The first proposition compares $\Delta_n(p,p')$ for different values of $n$.
This is the main innovation of our argument. 

\begin{proposition}\label{prop:comparison scale Delta}
    There exist constants $c_1,C_1>0$ such that for every $p'\ge p\ge p_c$ and every $16 n\le N\le \xi(p')$,
    \begin{equation*}
	    \Delta_N(p,p')\le c_1\qquad\Longrightarrow \qquad \Delta_n(p,p')\le C_1(\tfrac nN)^{c_1}.
    \end{equation*}
\end{proposition}

In words, if crossing probabilities do not vary too much at some scale $N$ as the parameter increases from $p$ to $p'$, 
then they do not vary at all at smaller scales $n \ll N$.  
 
The second proposition relates $\Delta_n(p,p')$ and the probability that $(\omega_p,\omega_{p'}) \in\calE_n$. 

\begin{proposition}\label{prop:comparison delta Delta}
	There exists $C_0>0$ such that for every $p'\ge p\ge p_c$ and every $n\le\xi(p')$,
	\begin{equation*}
		\Delta_{n}(p,p') \le \mathbf P[(\omega_p,\omega_{p'})\in\calE_n]\le C_0\,\Delta_{4n}(p,p').
	\end{equation*}
\end{proposition}

One can think of the proposition above as an analogue of the results relating pivotality for different events, which are central in \cite{Kes87}. Indeed, the first inequality is immediate, as the event in the definition of $\Delta_{n}(p,p')$ induces the existence of a mixed pivotal at scale~$n$. The second inequality should be understood as stating that a mixed pivotal may be transformed with positive probability to produce the event in the definition of $\Delta_n(p,p')$.
This result is a technical, but fairly standard consequence of the box-crossing property \eqref{eq:BXP}.

We now give the proof of the stability theorem based on the two propositions above.

\begin{proof}[Theorem~\ref{thm:stability}]
  Fix $p\ge p_c$ and write $\xi=\xi(p)$ for the correlation length at $p$. The first inequality $\theta_N(p_c)\le \theta_N(p)$ in \eqref{eq:stability pi1} is a consequence of monotonicity, hence we can focus on the second inequality $\theta_N(p)\le C \theta_N(p_c)$.

Our first step is to apply Proposition~\ref{prop:comparison scale Delta} to $N=\xi$ and prove that the crossing probabilities at scales $n\ll \xi$ are stable.
One cannot directly apply the proposition to the parameters $p_c$ and $p$ since $\Delta_\xi(p_c,p)$ is not particularly small. 
Nevertheless, one can easily circumvent this difficulty by decomposing the interval $[p_c,p]$ into smaller intervals, where Proposition~\ref{prop:comparison scale Delta} applies. We do so now.

Let $k\ge\frac1{c_1}$, where $c_1>0$ is  given by  Proposition~\ref{prop:comparison scale Delta}. 
Using that the function $f(q)= \mathrm P_{q}[\mathcal C(V_{\xi})]$ is nondecreasing and continuous in $q\in [p_c,p]$, 
we may choose $p_c=p_0\le p_1\le \dots\le p_k=p$ such that for every $0\le i<k$, 
$$f(p_{i+1})-f(p_i)=\frac1k [f(p)-f(p_c)]\le c_1.$$  
Since $\Delta_{\xi}(p_i,p_{i+1})\le f(p_{i+1})-f(p_i)\le c_1$, 
Proposition~\ref{prop:comparison scale Delta} applied to $N=\xi$ implies that 
\begin{equation*}
	\forall n\le \tfrac{1}{16}\xi,\qquad  \Delta_n(p_i,p_{i+1})\le C_1(\tfrac{n}{\xi})^{c_1}.
\end{equation*}
Proposition~\ref{prop:comparison delta Delta} thus gives that for every $0\le i<k$ and $n\le \tfrac1{16}\xi$, 
\begin{equation*}
	\mathbf P[(\omega_{p_i},\omega_{p_{i+1}})\in\calE_n]\le C_0C_1(\tfrac{4n}{\xi})^{c_1}.
\end{equation*}

Observe that $(\omega_{p_c},\omega_{p})\in\calE_n$ implies\footnote{To fully justify this fact, consider the smallest $q\in[p_c,p]$ such that $(\omega_{p_c},\omega_{q})\in\calE_n$; let $e$ be an edge such that $U(e)=q$. Then $(\omega_{p_i},\omega_{p_{i+1}})\in\calE_n$ for $i$ such that $p_i < q\le p_{i+1}$, with $e$ being the mixed pivotal.} 
the existence of $0\le i<k$ such that $(\omega_{p_i},\omega_{p_{i+1}})\in\calE_n$. 
Thus, for every $n\le \tfrac1{16}\xi$, using the inequality above for each $i$, we find 
\begin{equation}\label{eq:super}
	\mathbf P[(\omega_{p_c},\omega_{p})\in\calE_n]\le kC_0C_1(\tfrac{4n}{\xi})^{c_1}\le C_2(\tfrac{n}{\xi})^{c_1},
\end{equation}
where $C_2$ is a constant independent of $p$ and $n \leq \frac1{16}\xi$. 

Let us now fix $N\le \xi$, and let $c>0$ denote a small constant to be chosen later. 
For $0$ to be connected to $\partial\Lambda_N$ in $\omega_p$, either $0$ is already connected to $\partial\Lambda_{c N}$ in $\omega_{p_c}$, or the event  
\begin{equation*}
	\{0\text{ is connected to }\partial\Lambda_{n}\text{ in }\omega_p\}\cap \{(\omega_{p_c},\omega_p)\in \calE_n(x)\}\cap\{\Lambda_{8n}\text{ is connected to }\partial\Lambda_N\text{ in }\omega_p\}
\end{equation*}
 occurs for some scale $n \le cN$ of the form $n = 2^j$ with $j$ integer
 and some $x\in n\mathbb Z^2\cap \partial \Lambda_{3n}$, where $\calE_x(n)$ is the translate by $x$ of the event $\calE_n$.
Indeed, if $0$ is connected to $\partial\Lambda_N$ in $\omega_p$, but not to $\partial\Lambda_{c N}$ in $\omega_{p_c}$, consider a path contained in $\omega_p$, connecting $0$ to $\partial \Lambda_N$  and containing a minimal amount of edges not in $\omega_{p_c}$.  
Then any such edge is contained in $\La_{4n} \setminus \La_{2n}$ for some $n =2^j$ and is a mixed pivotal at scale $n$.
 
We deduce from the previous observation and independence that 
\begin{equation}
	\theta_N(p)\le \theta_{c N}(p_c)+\sum_{n=2^j\le c N} \theta_{n}(p)\mathbf P[(\omega_{p_c},\omega_{p})\in\calE_n]\theta_{8n,N}(p),
\end{equation}
where $\theta_{8n,N}(p)$ is the probability that $\Lambda_{8n}$ is connected to $\partial\Lambda_N$ in $\omega_p$.
 The quasi-multiplicativity of the one arm probability (Lemma~\ref{lem:quasi}) together with \eqref{eq:super} imply that
 \begin{equation}\label{eq:ahah}
	\theta_N(p)
	\le \theta_{c N}(p_c)+\!\!\!\!\!\underbrace{C_3\big(\tfrac{c N}{\xi}\big)^{c_1}}_{\le \frac12 \text{ if $c$ is small enough}}\!\!\!\!\!\!\!\!\theta_N(p).
\end{equation}
By choosing the constant $c$ small enough, we obtain $\theta_N(p)\le 2\theta_{cN}(p_c)$ and the result follows by applying the quasi-multiplicativity of the one arm probability (Lemma~\ref{lem:quasi}) again to compare $\theta_{cN}(p_c)$ to $\theta_N(p_c)$. This concludes the proof.
\end{proof} 

\begin{remark}\label{rem:stability4arm}
	Theorem~\ref{thm:stability} suffices to prove all scaling relations of \cite{Kes87} except for that linking the correlation length to the so-called four-arm probability. 
	For this final scaling relation, two elements are needed: the stability of the four arm probability under the correlation length 
	and linking $\Delta_N(p,p')$ to the four arm probability.
	The first may be deduced from~\eqref{eq:super} using the same argument as above. 
	Indeed, for a four arm event at distance $n$ to appear or disappear between $\omega_{p_c}$ and $\omega_p$, an event of the type $\calE_{2^j}$ with $j \leq \log_2n$ needs to occur. 
	The second ingredient is essentially equivalent to Russo's formula, and we will not detail it here. 
\end{remark}

\section{Proof of Proposition~\ref{prop:comparison scale Delta}}\label{sec:pivotal_switches}

The proof relies on the following induction relation on $\Delta_n(p,p')$ (a similar induction relation was used in \cite{DMT20} for a different quantity).

\begin{lemma}\label{lem:5}
	There exists a constant $c>0$ such that for every $p'\ge p \ge p_c$ and every $n,N\le \xi(p)$ satisfying $16 n\le N$,  we have
	\begin{equation}\label{eq:5}
		\Delta_N(p,p')\ge c  \left(\tfrac N n\right)^c\min\{\Delta_n(p,p'),(\tfrac nN)^2\}.
	\end{equation}
\end{lemma}

The above is the heart of the proof of Proposition~\ref{prop:comparison scale Delta}.
Before proving the lemma, we will show how it implies the proposition.

\begin{proof}[Proposition~\ref{prop:comparison scale Delta} subject to Lemma~\ref{lem:5}]
Fix $p'\ge p\ge p_c$ and write $\Delta_k$ instead of $\Delta_k(p,p')$. 
Also, fix an integer $R>16$ with $cR^c\ge 2$, where $c>0$ is the constant given by Lemma~\ref{lem:5}. 

Fix $n$ and $N$ as in the proposition. 
Let $K$ be the largest integer $k$ such that $16nR^{k}\le N$. Then,  \eqref{eq:5} implies that
\begin{align*}
&\Delta_N\ge c \, \min\{\Delta_{R^Kn},(16R)^{-2}\},\text{ and}\\
&\Delta_{R^kn}\ge 2\min\{\Delta_{R^{k-1}n},R^{-2}\}\text{ for every $1\le k\le K$}.
\end{align*}
As soon as $\Delta_N\le \frac{c}{(16R)^{2}}$, we may apply the previous relations recursively to obtain
\begin{equation}
\Delta_{n}\le 2^{-K}\Delta_{R^Kn}\le \tfrac{2^{-K}}{c}\Delta_N\le C_1(\tfrac nN)^{c_1}\Delta_N
\end{equation}
for some constants $c_1=c_1(R)>0$ and $C_1=C_1(c,c_1)$. This concludes the proof.
\end{proof}

The rest of the section is dedicated to proving Lemma~\ref{lem:5}. 
Fix two parameters $p'\ge p\ge p_c$ and two scales $n,N\le \xi(p)$ satisfying $n\le \tfrac1{16} N$.  
As above, we simply write $\Delta_k$ for $\Delta_k(p,p')$. 
All constants below are independent of the parameters $p$, $p'$, $n$ and $N$. 

 The proof relies on the notion of \emph{switches}, which correspond to  particular boxes  of size $n$ surrounded by  two paths in $\omega_p$  alternating with  two dual  paths in  $\omega_p^\star$. Even though the notion of switch is defined in term of the configuration at $p$, a key idea in the proof is to examine the effect of rising $p$ to $p'$ ``inside the switch''. 

Set $S:=(8n+1)\mathbb Z^2\cap \Lambda_{N/2}$. The left diagram of Fig.~\ref{fig:1} illustrates the definition below.  For $x\in S$, consider, when they exist,
\begin{itemize}[noitemsep]
\item the right-most path $\mathbf \Gamma_x^r$ of $\omega_p$ from bottom to top in $x+[n,2n]\times[-2n,2n]$,
\item the left-most path $\mathbf \Gamma_x^\ell$ of $\omega_p$ from bottom to top in $x+[-2n,-n]\times[-2n,2n]$,
\item   the bottom-most path ${\mathbf \Gamma}_x^{\star b}$ of $\omega_p^{\star}$ from ${\mathbf \Gamma}^\ell_x$ to ${\mathbf \Gamma}^r_x$ in $x+[-2n,2n]\times[-2n,-n]$,
\item the top-most path  ${\mathbf \Gamma}_x^{\star t}$  of $\omega_p^\star$ from $\mathbf \Gamma_x^\ell$ to $\mathbf \Gamma_x^r$ in $x+[-2n,2n]\times[n,2n]$. 
\end{itemize}
If the four paths above exist, we say that $x$ is a \emph{switch} and we write  $\mathbf Q_x$  for the region enclosed by the previous paths. Otherwise, we set $\mathbf Q_x=\emptyset$. 

\begin{figure}
\begin{center}
\raisebox{1cm}{\includegraphics[height = 5cm]{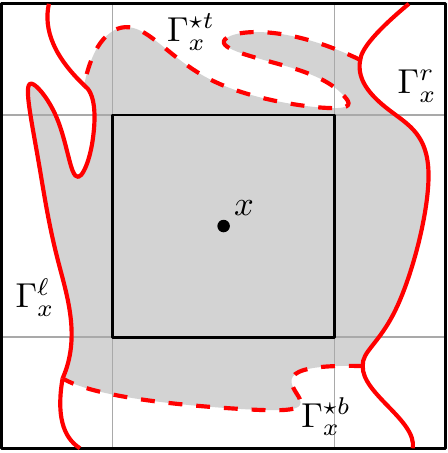}}
\hspace{1cm}\includegraphics[height = 7cm]{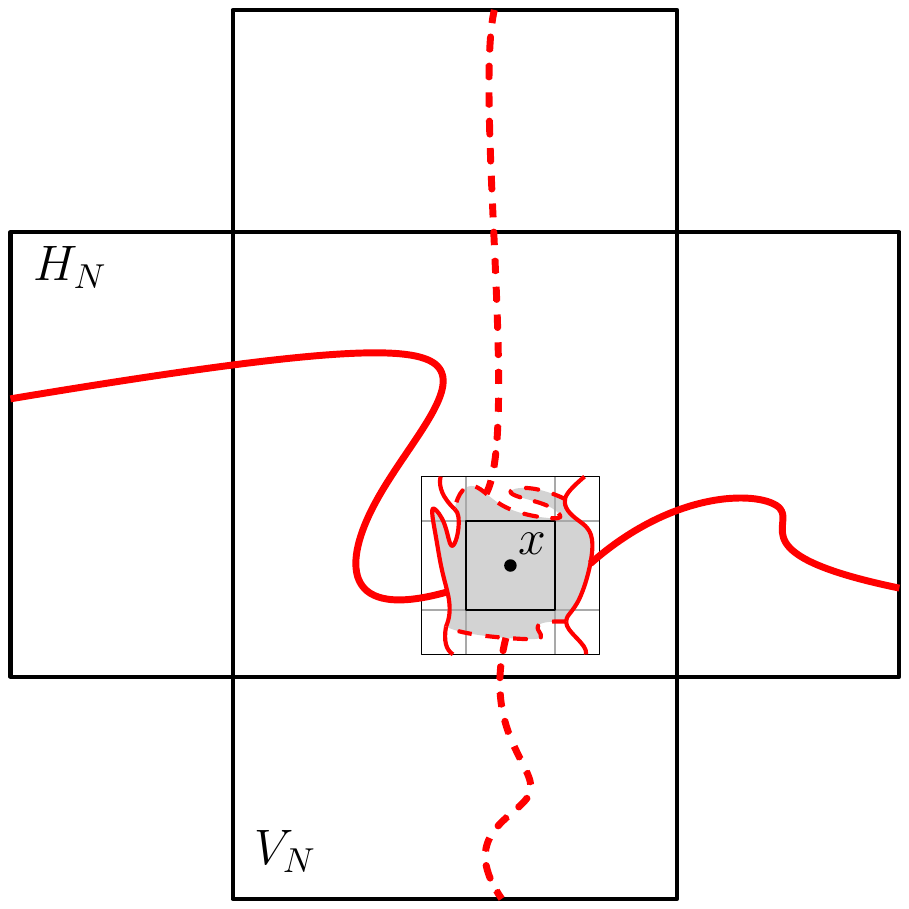} \caption{{\em Left:} The four paths $\mathbf \Gamma_x^t$, $\mathbf \Gamma_x^\ell$, ${\mathbf \Gamma}_x^{\star b}$ and ${\mathbf \Gamma}_x^{\star t}$ appearing in the definition of a switch at $x$. The gray region is $\mathbf Q_x$. 
{\em Right:} The four paths connecting the sides of $\mathbf Q_x$ to those of $H_n$ and $V_n$ ensure that $x$ is a pivotal switch.} 
\label{fig:1}
\end{center}\end{figure}

We now introduce the notion of pivotal switches, which resembles the notion of pivotal box in standard arguments dealing with critical Bernoulli percolation. See the right diagram of Fig.~\ref{fig:1} for an illustration.

\begin{definition}[Pivotal switches]
    A vertex  $x\in S$ is said to be a {\em pivotal switch} if 
    \begin{itemize}[noitemsep]
    \item $x$ is a switch (i.e. $\mathbf Q_x\ne \emptyset$);
    \item $\mathbf \Gamma_x^\ell$ and  $\mathbf \Gamma_x^r$ are  respectively connected to the left and right sides of $H_N$ in $\omega_p\cap H_N$;
    \item $\mathbf \Gamma_x^{\star t}$ and  $\mathbf \Gamma_x^{\star b}$ are respectively connected to the top and bottom sides of $V_N$ in $\omega_p^\star\cap V_N$.
    \end{itemize}
\end{definition}

Notice that when the connections between $\mathbf \Gamma_x^\ell,\dots, \mathbf \Gamma_x^{\star b}$
and the sides of $H_N$ and $V_N$ in the definition above occur, such connections necessarily occur outside of $\mathbf Q_x$, 
as depicted in Fig.~\ref{fig:1}. 
As such, the fact that $x$ is a pivotal switch is measurable in terms of $\omega_p$ outside of $\mathbf Q_x$.
Finally, by the choice of $S$, the regions $(\mathbf Q_x: x\in S)$ are pairwise disjoint and contained in $\La_N$. 

We also introduce the  notions of  $p$-open and $p'$-open  switches, which will lead to  the macroscopic equivalent of a pivotal edge being $p$-open and $p'$-open. The notions of switches and pivotal switches only involve the configuration at the parameter $p$. 
In contrast, it will be important to consider a notion of open switches that depends on the parameter. 

\begin{definition}[open/closed switches]
 Let $q\in \{p,p'\}$. A vertex $x\in S$ is said to be a  {\em $q$-open} (resp.~$q$-{\em closed}) \emph{switch} if 
\begin{itemize}[noitemsep]
\item $x$ is a switch (i.e. $\mathbf Q_x\ne \emptyset$);
\item $\mathbf \Gamma_x^\ell$ and $\mathbf \Gamma_x^r$ are connected (resp.~not connected) in $\omega_q\cap{\mathbf Q_x}$. 
\end{itemize}
\end{definition}

We can now focus on the proof of Lemma~\ref{lem:5}. The argument relies on the following observation: if there exists a pivotal switch which turns from being  $p$-closed to being $p'$-open, then the ``cross''-event defining $\Delta_N$ must occur, see~\eqref{eq:4}. In order to obtain a quantitative statement we introduce 
\[
\mathbf X:=\text{number of $p$-closed pivotal switches in $S$}\] 
and state the most important claim of the proof: 
\begin{claim}\label{lem:recursive relation} We have
\begin{equation}
\Delta_N\ge 1-\mathbf E[(1-\Delta_n)^{\mathbf X}].
\end{equation}
\end{claim}
\begin{proof}

Let $\mathcal F$ be the $\sigma$-algebra generated by the $\omega_p(e)$ with $e\notin \bigcup_{x\in S}\mathbf Q_x$ and the information of whether the switches in $ S$ are $p$-closed or $p$-open.
We make three easy observations:
\begin{itemize}[noitemsep]
\item[(1)] $\mathbf X$ is $\mathcal F$-measurable;
\item[(2)]  if there exists a pivotal switch $x$ that is $p$-closed  and $p'$-open,
then $\omega_{p'}\in \mathcal C(H_n)$ and  $\omega_p\notin \mathcal C(V_n)$;
\item[(3)] when conditioning on $\mathcal F$, each $p$-closed  switch $x$ is $p'$-open  independently of the other switches, with a probability\footnote{The first inequality is obtained by forgetting the denominator and by interpreting the difference of probabilities in terms of the increasing coupling. The second inequality is due to the inclusion of events. }
\[
\frac{\mathrm P_{p'}[\mathcal C(\mathbf Q_x)]-\mathrm P_{p}[\mathcal C(\mathbf Q_x)]}{1-\mathrm P_p[\mathcal C(\mathbf Q_x)]}\ge \mathbf P[\omega_{p'}\in \mathcal C(\mathbf Q_x),\omega_p\notin\mathcal C(\mathbf Q_x)]\ge \Delta_n
,\]
where $\mathcal C(\mathbf Q_x)$ denotes the event that $\mathbf \Gamma^\ell_x$ and $\mathbf \Gamma^r_x$ are connected to each other in $\mathbf Q_x$. 
\end{itemize}
The claim follows readily by first conditioning on $\mathcal F$ and then combining the previous three observations.
\end{proof}

One may recognise in \eqref{lem:recursive relation} a recursive relation similar to that obtained when studying the survival probability of branching processes. It therefore comes as no surprise that a lower bound on the expectation of $\mathbf E[\mathbf X]$ is an important piece of information to use in conjunction with the previous claim. We derive such a lower bound in the next claim. Note that this is a statement at $p$ only.

\begin{claim}\label{claim:lower bound}
There exists $c>0$ independent of $p, n, N$ such that
\begin{equation}
\mathbf E[\mathbf X]\ge c(\tfrac Nn)^c.
\end{equation}
\end{claim}

 \begin{figure}[t]
  \begin{center}
   \includegraphics[width = .45\textwidth]{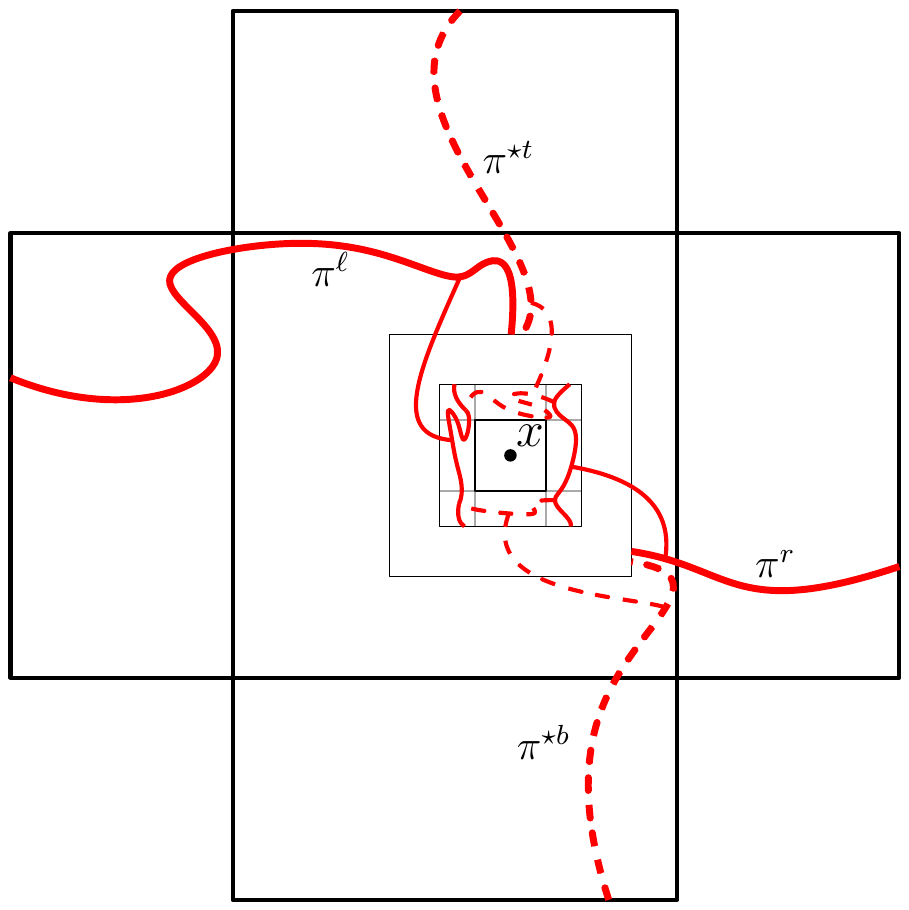}
  \hspace{.05\textwidth}
 \includegraphics[width = .45\textwidth]{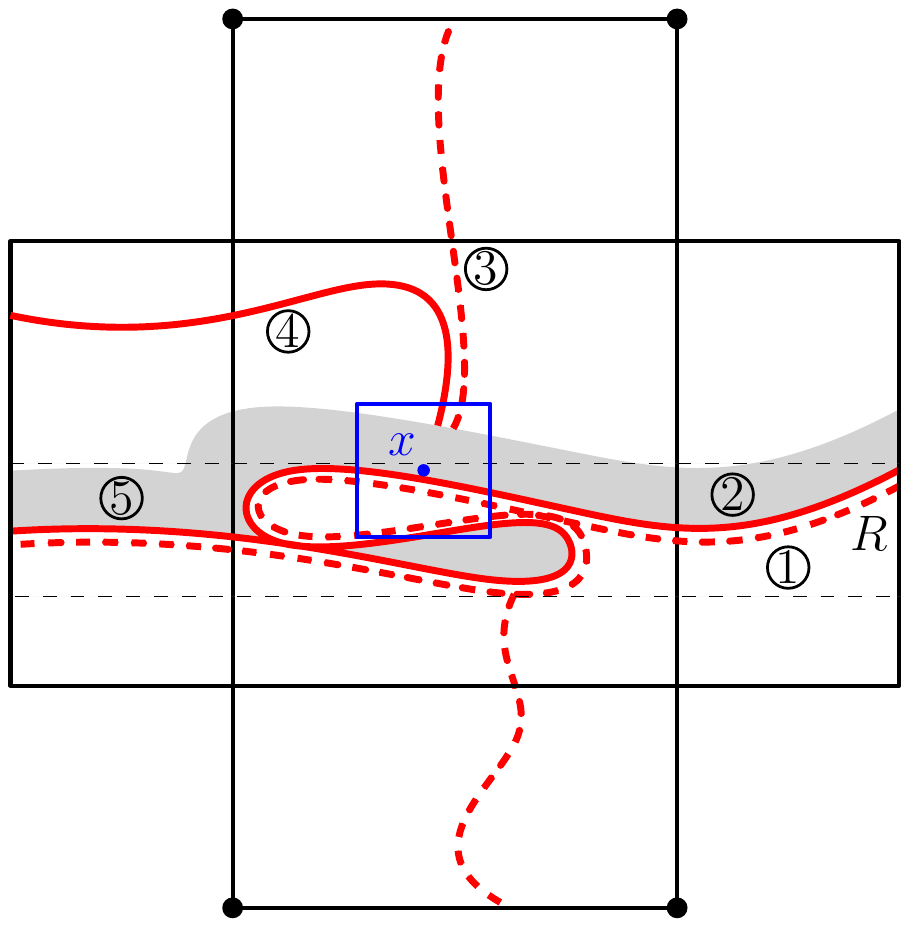}
  \caption{{\em Left:} When $\calA_4^{\rm sep}(x)$ occurs, 
  if we define $\pi^\ell$ to be the top-most connection between the left side of $H_N$ and $\La_{8n}(x)$,  
  $\pi^{\star t}$ to be the left-most dual connection between the top of $V_N$ and $\La_{8n}(x)$,  
  $\pi^r$ to be the bottom-most connection between the right side of $H_N$ and $\La_{8n}(x)$ and  
  $\pi^{\star b}$ to be the right-most dual connection between the bottom of $V_N$ and $\La_{8n}(x)$, 
  then these paths necessarily satisfy the separation requirement. 
  The space left between them may be used to connect them to primal and dual crossings 
  of the $4n\times n$ rectangles contained in $\La_{2n}\setminus \La_n$, thus producing ${\mathcal Piv}(x)$.
  Due to the box-crossing property \eqref{eq:BXP}, the latter connections and crossings occur with uniformly positive probability, 
  conditionally on $\pi^\ell$, $\pi^{\star t}$,  $\pi^r$ and  $\pi^{\star b}$, thus proving \eqref{eq:aha}. \newline
  {\em Right:} A realisation of $\mathbf \Gamma$ contained in $R$; the domain $D_{\mathbf \Gamma}$ is above the shaded area. 
  When $E_{\mathbf \Gamma}$ occurs, there exists at least one $x \in S$ for which $\calA_4^{\rm sep}(x) \circ \calA_1(x)$ occurs: the arms producing $\calA_4^{\rm sep}(x)$ for the blue box are numbered $1,\dots,4$; the fifth arm is disjoint from the first four and produces $\calA_1(x) $.
    }\label{fig:pivs}
  \end{center}
  \end{figure}

	The result may be obtained through classical arguments of planar percolation: 
	the separation of arms allows one to compare $\mathbf E[\mathbf X]$ to the expected number of pivotal boxes, 
	which in turn may be shown to be polynomially high. 
	We give below a modified version of the argument, 
	which short-circuits the separation of arms by directly showing a polynomial bound for the number of pivotal boxes with separated arms. 
\begin{proof}   	
  For $x \in S$, let $\calA_4^{\rm sep}(x)$ be the event that there exist 
  two primal paths $\pi^\ell$ and $\pi^r$ in $\omega_p\cap H_N$ respectively from the left and right of $H_N$ reaching $\Lambda_{8n}(x)$ 
  and two dual paths $\pi^{\star t}$ and $\pi^{\star b}$ in $\omega^\star_p \cap V_N$ from the top and bottom of $V_N$ reaching $\Lambda_{8n}(x)$ 
  such  that the Euclidean distance between $\pi^\ell\cup \pi^{\star t}$ and  $\pi^r\cup \pi^{\star b}$ is at least $2n$.
  
    A construction based on the box-crossing property \eqref{eq:BXP} and illustrated in Fig.~\ref{fig:pivs} on the left (see also the caption for details) shows	the existence of a constant $c_0$ independent of $p$, $n$ and $N$ such that, for every $x\in S$, 
    \begin{equation}\label{eq:aha}
	    \mathbf P[{\mathcal Piv}(x)]\ge c_0 \mathbf P[\calA_4^{\rm sep}(x)],
    \end{equation}
    where ${\mathcal Piv}(x)$ is the event that $x$ is a $p$-closed  pivotal switch.

Therefore, it suffices to prove the existence of a constant $c'>0$ such that
\begin{equation}\label{eq:h}
    \sum_{x\in S} \mathbf P[\mathcal A_4^{\rm sep}(x)]\ge c'(\tfrac Nn)^{c'}.
\end{equation}

Consider the ``cross''-shaped domain $C:=H_N\cup V_N$. Define the top and bottom of $C$ to be respectively the top and bottom sides of $V_N$. The left and right sides of $C$ are the remaining two parts of the boundary of $C$, both being the union of five segments, see Fig~\ref{fig:pivs} on the right. Define $\mathbf \Gamma$ to be the lowest path of $\omega_p\cap C$ from the left to the right side of~$C$;
set $\mathbf \Gamma = (0,-\infty)$ when no such path exists. 
Let $R:=[-2N,2N]\times[-N/2,0]$ be a thin horizontal rectangle inside~$H_N$. 
As a direct consequence of the box-crossing property \eqref{eq:BXP}, there exists a uniform constant $c_1 > 0$ such that
\begin{equation}\label{eq:6}
	  \mathbf P[\mathbf \Gamma \subset R]\ge c_1.
\end{equation}

For every realisation $\gamma \subset R$ of $\mathbf \Gamma$, 
consider the domain $D_\gamma \subset C$, defined as the set of points of $C$ above $\gamma$ 
that are at a Euclidean distance at least $2n$ from $\gamma$. 
Define the event $E_\gamma$ that there exist a primal vertex $u$ and a dual vertex $v$ both on the intersection of the boundary of $D_\gamma$ and $ \Lambda_{N/2}$ 
with $|u-v|=1/2$ and
\begin{itemize}
	\item $u$ is connected to the left side of $H_N$ in $\omega_p\cap D_\gamma$,
	\item $v$ is connected to the top side of $V_N$ in $\omega_p^\star\cap D_\gamma$.
\end{itemize}
Again, using a standard construction and the box-crossing property, 
one can prove that $\mathbf P[E_\gamma]\ge c_2$ for every $\gamma$ as above, where $c_2$ is a uniformly positive constant.

Now, for every realisation $\gamma\subset R$ of $\mathbf \Gamma$, 
the events $\{\mathbf \Gamma=\gamma\}$ and $E_\gamma$ are independent. Thus
\begin{align}
 \nonumber
  \mathbf P[ \mathbf \Gamma \subset R \text{ and $E_{\mathbf \Gamma}$ occurs}]
  &= \sum_{\gamma\subset R} \mathbf P[\{\mathbf \Gamma=\gamma\}\cap E_\gamma]\\
  &\ge c_2\mathbf P[\mathbf \Gamma \subset R]\ge c_1c_2.      \label{eq:9}
\end{align}

For $x \in S$, let $\mathcal A_1(x)$ denote the event that there exists a path in $\omega_p\cap C$ from $\Lambda_{8n}(x)$ to the left side of $R$. When  $\mathbf \Gamma$ is contained in $R$ and $E_{\mathbf \Gamma}$ occurs, there exists at least one vertex $x\in S$ such that $\mathcal A_4^{\rm sep}(x)$ and $\mathcal A_1(x)$ occur disjointly -- we denote this event by $\mathcal A_4^{\rm sep}(x)\circ\mathcal A_1(x)$.
Indeed, the event $\mathcal A_4^{\rm sep}(x)\circ\mathcal A_1(x)$ occurs for any $x$ such that the points $u,v$ appearing in the definition of $E_{\mathbf \Gamma}$ are contained in $\La_{8n}(x)$; see the right diagram of Fig.~\ref{fig:pivs} for details. 
Therefore, the union bound together with \eqref{eq:9} imply that
\begin{equation}\label{eq:10}
  c_1c_2\le \sum_{x\in S} \mathbf P[\mathcal A_4^{\rm sep}(x)\circ \mathcal A_1(x)].   
\end{equation}
The  box-crossing property \eqref{eq:BXP} implies that  $\mathbf P[\mathcal A_1(x)]\le (n/N)^{c_3}$ for every $x\in S$. Therefore, by applying\footnote{
The derivation of \eqref{eq:11} from \eqref{eq:9} can be done without the use of Reimer's inequality, but we chose to present the shortest argument.} Reimer's inequality \cite{Rei00} in the equation above,   we conclude that
\begin{equation}
  \label{eq:11}
   c_1c_2\le \left(\tfrac n N\right)^{c_3} \sum_{x\in S} \mathbf P[\mathcal A_4^{\rm sep}(x)].
\end{equation}
This concludes the proof of \eqref{eq:h} and therefore of the whole claim.
\end{proof}

Finally, Lemma~\ref{lem:5} directly follows from the two claims above. 

\begin{proof}[Lemma~\ref{lem:5}]
	Using that $\mathbf X\le |S|\le (N/n)^2$ and 
    that $x\mapsto 1-(1-x)^y$ is decreasing in $x$ and larger than $\tfrac1eyx$ when $0\le yx\le 1$,
    Claim~\ref{lem:recursive relation} yields
    \begin{equation}
	    \Delta_N\ge \tfrac1e\mathbb E_p[\mathbf X] \min\{\Delta_n,(\tfrac nN)^2\}.
    \end{equation}
    Then, simply plug the lower bound of Claim~\ref{claim:lower bound} in the equation above to obtain \eqref{eq:5}.
\end{proof}

\section{Proof of Proposition~\ref{prop:comparison delta Delta}}

As already mentioned, Proposition~\ref{prop:comparison delta Delta} is similar to previous results known as separation lemmas for arm events -- see for instance~\cite{Kes87,Nol08,GPS}. 
In our context, the event $\{\omega_{p'}\in \mathcal C(H_n),\omega_{p}\notin \mathcal C(V_n)\}$ plays the role of the well-separated arm event and $\calE_n$ that of an unconstrained arm event.  
The proof below follows the lines of the original argument of Kesten. The reader well acquainted with arm separation arguments may skip the whole section.
Three lemmas will come into play; the first states that well-separated arms may be extended from scale $n$ to scale $2n$ at constant cost. 
 
\begin{lemma}\label{lem:1}
    There exists $c_1>0$ such that for every $p'\ge p\ge p_c$ and $n\le \xi(p')$,
    \[
	   \Delta_{4n}(p,p')\ge c_1\,\Delta_{4\lfloor n/2\rfloor}(p,p').
    \] 
\end{lemma}

\begin{proof} Fix $p$ and $p'$ and drop them from notations.
  For simplicity, we prove $ \Delta_{2n}\ge c_1\,\Delta_{n}$; the statement in the lemma is obtained exactly in the same way.  Let \[A:=\{\omega_{p'}\in \mathcal C(H_n),\,\omega_{p}\notin \mathcal C(V_n)\}\] be the event appearing in the definition of $\Delta_n$.    Consider the event $B$ that in the configuration $\omega_{p'}$ there exist
   \begin{itemize}[noitemsep]
   \item vertical crossings of $[-2n,-n]\times[-n,n]$ and $[n,2n]\times[-n,n]$,
   \item horizontal crossings of $[-4n,-n]\times[-n,n]$ and $[n,4n]\times[-n,n]$.
   \end{itemize}
   When both $A$ and $B$ occur, then the event $A'$ defined in the same way as $A$ except that the horizontal rectangle $H_n$ is replaced by the longer rectangle $H_n':=[-4n,4n]\times[-n,n]$ must also occur.  
   
   In order to estimate the probability of $A\cap B$, condition on the   $\sigma$-algebra  $\mathcal F$  generated by the configuration $\omega_{p'}$ outside the rectangle $V_n$, and use the following two observations. First, the event $B$ introduced above is increasing and measurable with respect to $\mathcal F$. Second, the  conditional expectation $\mathbf E[\mathbf 1_A\:|\: \mathcal F]$ is increasing (as  a random variable defined on the product space $[0,1]^{\mathbb E}$). Using these two facts together with the FKG inequality (applied to the product  measure $\mathbf P$), we find
      \begin{align}\label{eq:21}
       \mathbf P[A']
       \ge  \mathbf P[A\cap B]
       = \mathbf E\big[\mathbf E[\mathbf 1_A\:|\: \mathcal F ] \cdot \mathbf 1_B\big ]
       \ge \mathbf E\big[\mathbf E[\mathbf 1_A\:|\: \mathcal F ]\big] \mathbf 	P[ B]
       =   \mathbf P[A] \mathbf P[ B].
      \end{align}
      The box-crossing property \eqref{eq:BXP} and the FKG inequality give $\mathbf P[B]\ge c_2$,
      for some constant $c_2 >0$ independent of $n$, $p$ or $p'$. 
     With $\Delta_n=\mathbf P[A]$, this implies that $\mathbf P[A']\ge c_2 \Delta_n$. 
      A similar construction for the dual configuration $\omega_p^\star$ which uses the FKG inequality for decreasing random variables 
      implies that $\Delta_{2n}\ge c_2 \mathbf P[A']\ge c_2^2 \Delta_n$.
\end{proof}

We now define a ``bad'' event which prevents the arms generated by $\calE_n$ from separating.
For $r \leq R \leq n$ and $x \in \partial \La_{2n}$, define the {\em mixed three arm event} $\calA_{\rm mix}(x; r,R)$ 
as the set of configurations $\omega \leq \omega'$ such that
$\La_{2n} \cap [\Lambda_{R}(x) \setminus\Lambda_{r}(x)]$ contains three disjoint paths (or \emph{arms}) $\gamma_1,\gamma_2, \gamma_3$ 
connecting $\Lambda_{r}(x)$ to $\partial \Lambda_{R}(x)$ and either: 
\begin{itemize}
	\item $\gamma_1$ is contained in $\omega$, $\gamma_2$ is contained in $\omega^\star$ and $\gamma_3$ is contained in $\omega'$ or 
	\item $\gamma_1$ is contained in $\omega'$, $\gamma_2$ is contained in $(\omega')^\star$ and $\gamma_3$ is contained in $\omega^\star$.
\end{itemize}
Furthermore, for $0<\eta_1 < \dots <\eta_8 \leq 1$, with $\eta_{i+1}/\eta_i >100$, for each $i$ and for $n \geq \eta_1^{-1}$ define the bad event\footnote{Readers familiar with arm separation may wonder why we chose such a complicated definition for the bad event $\calB_n$.
When working with two configurations rather than a single one, it is harder to characterise scenarios that force arms not to be separated. In particular, the authors believe that excluding the mixed three arm event at a single scale would not suffice to guarantee well separation of arms. Excluding three mixed arms at two scales rather than seven does suffice, but would render the proof of Claim~\ref{claim:2} more tedious. We therefore use seven scales as a technical artefact.}
\begin{align*}
\calB_n = \calB_n(\eta_1,\dots,\eta_8) := \bigcup_{\substack{1 \leq i \leq 7\\x\in \partial \La_{2n}}} \calA_{\rm mix}(x; 10\eta_i n, \eta_{i+1}n/10).
\end{align*}
In the above and henceforth we omit integer parts to lighten notation. 
The lower bound $ n \geq \eta_1^{-1}$ is meant to avoid degeneracies.

The second lemma provides an upper bound on the probability of bad events.

\begin{lemma}\label{lem:2}
    For every $\eps > 0$, there exists $0<\eta_1 < \dots <\eta_8 \leq 1$, with $\eta_{i+1}/\eta_i >100$ for each $i$ such that
   	for every $p'\ge p\ge p_c$ and $n\in [\eta_1^{-1},\xi(p')]$, 
    \[
	    \mathbf P[(\omega_p,\omega_{p'}) \in \mathcal B_n(\eta_1,\dots,\eta_8) ]\le \eps.
    \] 
\end{lemma}

\begin{proof}
We start with a claim on the probability of the mixed three arm event within the critical window. 
\begin{claim}\label{claim:1}
	There exist $\delta > 0$ and $C_2 > 0$ such that
	for every $p'\ge p\ge p_c$, $r\leq R\leq n\leq \xi(p')$ and $x\in \partial \La_{2n}$,
	\begin{align*}
		\mathbf P [(\omega_p,\omega_{p'}) \in \calA_{\rm mix}(x; r,R)] \leq C_2 (r/R)^{1 + \delta}.
	\end{align*}
\end{claim}

The bound above is a standard consequence of the box-crossing property in the critical window \eqref{eq:BXP} and of Reimer's inequality. We provide a proof for completeness at the end of the appendix.

	A direct consequence of Claim~\ref{claim:1} is that there exists a constant $C>0$ such that for all $p$, $p'$, $r$, $R$ and $n$
	as in the statement,
	\begin{align}\label{eq:existence_of_mixed}
		\mathbf P [\exists x \in \partial \La_{2n} \text{ with }(\omega_p,\omega_{p'}) \in \calA_{\rm mix}(x; r,R)] 
		\leq C \tfrac{n}{R}\big(\tfrac{r}{R}\big)^\delta.
	\end{align}	
	Indeed, it suffices to consider points $z_1,\dots, z_K \in \partial \La_{2n}$ distributed in counterclockwise order, 
	with each $z_i$ at a distance $r$ from $z_{i+1}$ (using notation modulo $K$), with $K \leq |\partial \La_{2n}|/r=O(n/r)$.
	If the above event occurs, then $(\omega_p,\omega_{p'}) \in \calA_{\rm mix}(z_i; 2r,R/2)$ for some $1 \leq i\leq K$. 
	Performing a union bound over $i$ gives the inequality \eqref{eq:existence_of_mixed}.
	
	Consider $\eps > 0$ a small constant. Let $\eta_8 = 1$ and $\eta_{i} = \eta_{i+1}^{1+1/\delta}\eps^{1/\delta}$ for every $i \leq 7$. 
	Then, due to \eqref{eq:existence_of_mixed},  for every $i$, 
	\begin{align*}
		\mathbf P [\exists x \in \partial \La_{2n} \text{ with }(\omega_p,\omega_{p'}) \in \calA_{\rm mix}(x; 10\eta_i n, \eta_{i+1}n/10)] 
		\leq C' \eps,
	\end{align*}
	with $C' = 10^{1 + 2\delta}C$.
	A union bound over $i = 1,\dots, 7$ gives	\begin{align*}
		\mathbf P [(\omega_p,\omega_{p'}) \in \calB_n(\eta_1,\dots, \eta_8)] \leq 7 C'\eps. 
	\end{align*}
	By choosing $\eps$ small enough, the above may be rendered arbitrarily small.
\end{proof}

Our final lemma states that when a mixed pivotal exists and the bad event fails to occur, 
the arms are well separated and the event in the definition of $\Delta_{4n}$ may be produced with positive probability. 

\begin{lemma}\label{lem:3}
	For every $0<\eta_1 < \dots <\eta_8 \leq 1$, with $\eta_{i+1}/\eta_i >100$, 
	there exists $C_3 = C_3(\eta_1,\dots,\eta_8) > 0$ such that, when these parameters are used to defined $\calB_n$, 
	for every $p'\ge p\ge p_c$ and  $n\in[\eta_1^{-1}, \xi(p')]$, 
	\begin{align*}
		\mathbf P[(\omega_p,\omega_{p'})\in \calE_n\setminus \mathcal B_n]\le C_3 \Delta_{4n}.
	\end{align*}
\end{lemma}

\begin{proof}

	We start with a claim.
	\begin{claim}\label{claim:2}
		For $(\omega,\omega') \in \calE_n\setminus \mathcal B_n$ there exist points $y_1,\dots, y_4 \in \partial \La_{2n}$ 
		distributed in counterclockwise order, at a distance of at least $\eta_2n$ from each other, such that
		\begin{enumerate}[label=(\roman*),noitemsep]
			\item\label{item:1} $y_1$ is connected to $y_3$ in  
			$\omega'\cap  \Lambda_{2n} \cap \Lambda_{3\eta_1n}(y_2)^c\cap  \Lambda_{3\eta_1n}(y_4)^c$ and
			\item\label{item:2} $y_2$ is connected to $y_4$ in  
			$\omega^\star \cap  \Lambda_{2n} \cap \Lambda_{3\eta_1n}(y_1)^c\cap  \Lambda_{3\eta_1n}(y_3)^c$.
		\end{enumerate}  
	\end{claim}
	
	This is a purely deterministic statement whose proof can be skipped in a first reading. 
	\medbreak
	\noindent{\em Proof of Claim~\ref{claim:2}:}
		Consider $(\omega,\omega') \in \calE_n\setminus \mathcal B_n$, 
		and paths $\gamma$ and $\gamma^\star$ as in the definition of~$\calE_n$. 
		Write $x_1,x_3 \in \partial \La_{2n}$ for the endpoints of $\gamma$ and $x_2,x_4 \in \partial \La_{2n}$ for the endpoints of $\gamma^*$.
		
		To start, observe that if $y_1,\dots, y_4 \in \partial \La_{2n}$ may be chosen in counterclockwise order, 
		at a distance of at least $\eta_2n$ from each other, and such that $y_1$ is connected to $y_3$ in  
		$\omega'\cap  \Lambda_{2n}$ and $y_2$ is connected to $y_4$ in $\omega^\star \cap  \Lambda_{2n}$, then 
		such connections may be chosen to avoid $\Lambda_{3\eta_1n}(y_2)\cup  \Lambda_{3\eta_1n}(y_4)$ and 
		$\Lambda_{3\eta_1n}(y_1)\cup  \Lambda_{3\eta_1n}(y_3)$, respectively. 
		Indeed, the absence of the mixed three arm events between 
		$\Lambda_{10\eta_{1}n}(y_i)$ and $\partial \Lambda_{\eta_{2}n/10}(y_i)$ for $i = 2$ and $i = 4$
		implies that the connection between $y_1$ and $y_3$ may not be forced to visit $\Lambda_{3\eta_1n}(y_2)\cup  \Lambda_{3\eta_1n}(y_4)$.
		The same holds for the connection between $y_2$ and $y_4$. 
		
		Thus, it suffices to find $y_1,\dots, y_4$ at mutual distances of at least $\eta_2n$ from each other, 
		connected in $\omega^\star$ and $\omega'$, respectively, inside $\La_{2n}$. 
		We will do so based on the points $x_1,\dots, x_4$ above. It will be useful to keep in mind that the connection between $y_1$ and $y_3$ 
		and that between $y_2$ and $y_4$ are allowed to intersect in arbitrary ways. 
		
		We consider several situations (for each case, we also consider any permutation of it): 
		\begin{enumerate}[label=(\roman*),noitemsep]
			\item\label{item:1} the mutual distances between $x_1,\dots, x_4$ are all larger than $\eta_2 n$;
			\item\label{item:1} $\|x_1 - x_2\| \leq \eta_2 n$, but $\|x_1 - x_3\| > \eta_3 n$, $\|x_1 - x_4\| > \eta_3 n$ and  $\|x_3 - x_4\| > \eta_2 n$;
			\item\label{item:1} $\|x_1 - x_2\| \leq \eta_2 n$ and $\|x_3 - x_4\| \leq \eta_2 n$, 
			but $\|x_1 - x_3\| > \eta_4 n$;
			\item\label{item:1} $x_2,x_3 \in \La_{\eta_4 n}(x_1)$, 
			but $\|x_1 - x_4\| \geq \eta_5 n$;
			\item\label{item:1} $x_2,x_3,x_4 \in \La_{\eta_5 n}(x_1)$.
		\end{enumerate}  
		It is immediate that the above cases, together with their permutations, cover all possible settings for the distribution of $x_1,\dots, x_4$ along the boundary of $\La_{2n}$. 
		
		In the first case, we may simply choose $(y_1,\dots, y_4) = (x_1,\dots, x_4)$. In all other cases, some modification of $(x_1,\dots,x_4)$ is needed. We illustrate this in the last case, which is also the most complex.

		With no loss of generality, suppose that $x_1,\dots, x_4$ are distributed from left to right on $\partial \La_{2n} \cap \La_{\eta_5 n}(x_1)$;
		see the left diagram of Fig.~\ref{fig:bad} for an illustration. 
		Then in each annulus $\La_{\eta_{i+1} n/10}(x_1) \setminus \La_{10 \eta_i n}(x_1)$ with $i = 5,6,7$, 
		there exist disjoint subpaths of $\gamma_1,\dots, \gamma_4$ ordered from left to right, crossing from the inside to the outside of the annulus. 
		As the mixed three arm event does not occur in $\La_{\eta_{8} n/10}(x_1) \setminus \La_{10 \eta_7 n}(x_1)$, 
		there exists a path in this annulus contained in $\omega'$, connecting $\gamma_1$ to the left side of $\partial \La_{2n}$.
		Call the endpoint on $\partial \La_{2n}$ of such a path $y_1$. 
		The same argument shows that $\gamma_2$ is connected to a point $y_2$ on the left side of $\partial \La_{2n}$
		in $\La_{\eta_{7} n/10}(x_1) \setminus \La_{10 \eta_6 n}(x_1)$ by a path in $\omega^\star$
		and $\gamma_3$ is connected to a point $y_3$ on the left side of $\partial \La_{2n}$
		in $\La_{\eta_{6} n/10}(x_1) \setminus \La_{10 \eta_5 n}(x_1)$ by a path in $\omega'$.
		Finally, set $y_4= x_4$. The points $y_1,\dots,y_4$ thus obtained satisfy the desired conditions. 
		
		Similar constructions hold in the other cases to produce $y_1,\dots, y_4$. 
	\hfill $\diamond$\bigskip 
	
	\begin{figure}
        \begin{center}
            \includegraphics[width = 0.58\textwidth]{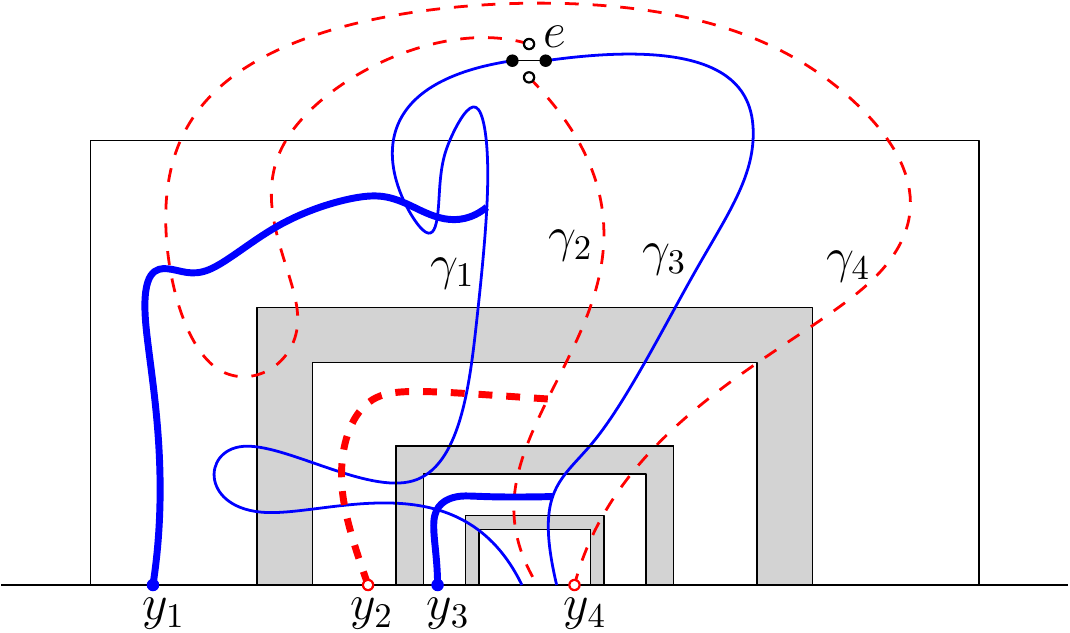}\hspace{.02\textwidth}
            \includegraphics[width = 0.39\textwidth]{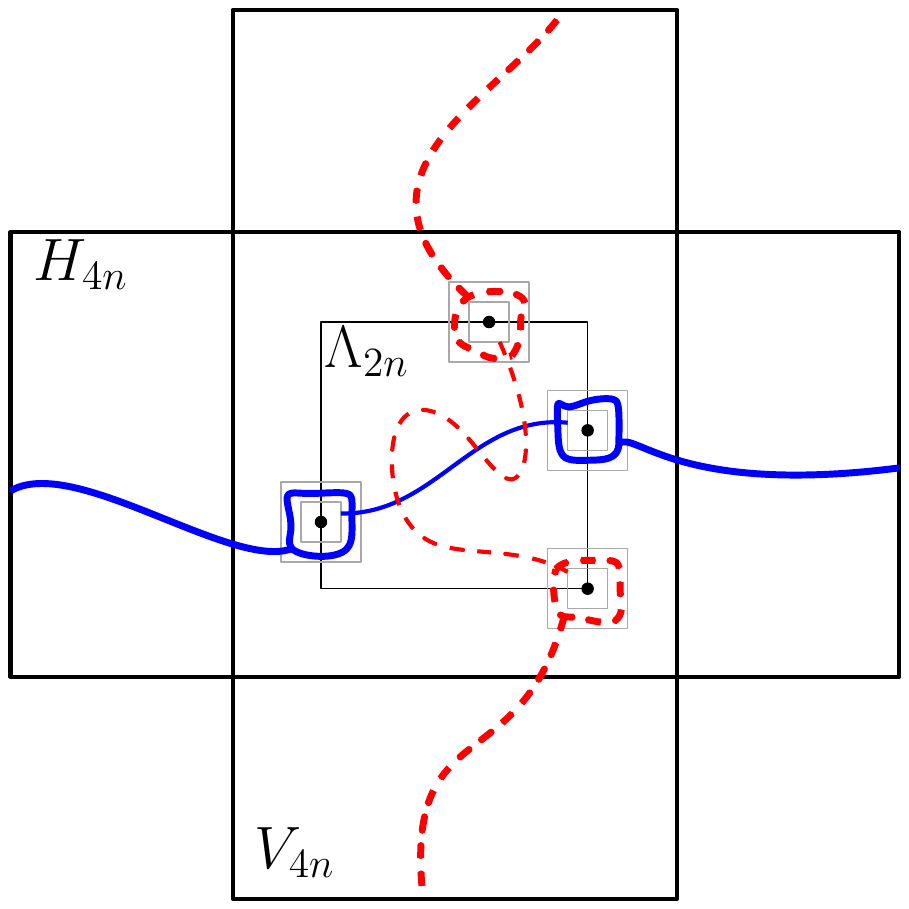}
            \caption{{\em Left:} The annuli $\La_{\eta_{i+1} n/10}(x_1) \setminus \La_{10 \eta_i n}(x_1)$ with $i = 5,6,7$;
             the buffer zones $\La_{10 \eta_{i} n}(x_1) \setminus \La_{\eta_i n/10}(x_1)$ are shaded.
            The paths connecting $\gamma_1$, $\gamma_2$ and $\gamma_3$ to the left-hand side of $\partial \La_{2n}$ 
            inside the respective annuli $\La_{\eta_{i+1} n/10}(x_1) \setminus \La_{10 \eta_i n}(x_1)$ are drawn in bold. 
            They exist due to the lack of mixed three arm events. \newline
            {\em Right:} The event $\calG(i,j,k,\ell)$ occurs due to the thin blue and dashed red paths, belonging to $\omega_{p'}$ and $\omega_p^\star$, respectively. The thicker blue and dashed red paths exist with constant probability due to the box-crossing property, and induce the event in the definition of $\Delta_{4n}$. }
            \label{fig:bad}
        \end{center}
		\end{figure}

	We now prove the lemma. 
	Consider $p$, $p'$ and $n$ as in the statement and $(\omega_p,\omega_{p'})\in \calE_n\setminus \mathcal B_n$.
	Fix points $z_1,\dots, z_K \in \partial \La_{2n}$ distributed in counterclockwise order, 
	with each $z_i$ at a distance of at most $\eta_1 n$ from $z_{i+1}$ (using notation modulo $K$), 
	with $K \leq |\partial \La_{2n}|/\eta_1 n=O(1)$.
	Then, due to Claim~\ref{claim:2}, there exist four points $z_i,z_j,z_k,z_\ell$ among the above, distributed in counter-clockwise order, 
	such that $\La_{\eta_1 n}(z_i)$ is connected to $\La_{\eta_1 n}(z_k)$ in 
	$\omega_{p'}\cap  \Lambda_{2n} \cap \Lambda_{2\eta_1n}(z_j)^c\cap  \Lambda_{2\eta_1n}(z_\ell)^c$, 
	while $\La_{\eta_1 n}(z_j)$ is connected to $\La_{\eta_1 n}(z_\ell)$ in 
	$\omega_{p}^\star\cap  \Lambda_{2n} \cap \Lambda_{2\eta_1n}(z_i)^c\cap  \Lambda_{2\eta_1n}(z_j)^c$.
	Write $\calG(i,j,k,\ell)$ for the existence of the connections above. 
	We conclude that 
	\begin{align}\label{eq:14}
		\mathbf P[(\omega_p,\omega_{p'})\in \calE_n\setminus \mathcal B_n]
		\le \sum_{i,j,k,\ell} \mathbf P[(\omega_p,\omega_{p'})\in \mathcal \calG(i,j,k,\ell)].
	\end{align}
	
	Using a similar argument to that in the proof of Lemma~\ref{lem:1}, 
	the box-crossing property \eqref{eq:BXP} and the FKG-inequality, 
	we deduce that for every quadruple $i,j,k,\ell$ as above, 
	\begin{equation*}
      \mathbf P[(\omega_p,\omega_{p'})\in\mathcal G (i,j,k,\ell)]\le C \Delta_{4n},
    \end{equation*}
    for some constant $C$ independent of $n$. See also the right diagram of Fig.~\ref{fig:bad} for an illustration. 
	Finally, since the number of terms in the sum in \eqref{eq:14} is bounded uniformly in $n$, we obtain the desired conclusion. 
\end{proof}

We are now ready to proceed with the proof of Proposition~\ref{prop:comparison delta Delta}.

\begin{proof}[Proposition~\ref{prop:comparison delta Delta}]
In the proof below, we drop the dependence in the parameters $p,p'$ to lighten our notation:  
we write $\Delta_n:=\Delta_n(p,p')$, 
$\calE_n:=\{(\omega_p,\omega_{p'})\in \calE_n\}$  
and $\delta_n:=\mathbf P[\calE_n]$. 
Even if the parameters are kept  implicit, it is important that the constants $c_1$ and $\eta_1\dots, \eta_8$ introduced in the proof below are independent of $p'\ge p\ge p_c.$   With this notation, our goal is to show that for every $n\le \xi(p')$ we have
\begin{equation}
  \label{eq:7}
  \delta_n\le C_0 \Delta_{4n},
\end{equation}
where $C_0$ is independent of $n$ and  $p'\ge p\ge p_c$.

Let $c_1>0$ be as in Lemma~\ref{lem:1}. 
By Lemma~\ref{lem:2}, we may fix $\eta_1,\dots, \eta_8 > 0$ such that for every $n \in [\eta_1^{-1},\xi(p')]$, we have
\begin{equation}\label{eq:8}
  \mathbf P[(\omega_p,\omega_{p'})\in \mathcal B_n(\eta_1,\dots, \eta_8)]\le \tfrac{c_1}8.
\end{equation}

Then, for every $n\in [\eta_1^{-1},\xi(p')]$, we have
\begin{equation}
  \label{eq:12}
  \delta_n=\mathbf P[\calE_n]= \mathbf P[\calE_n\setminus \mathcal B_n]+  \mathbf P[\calE_n\cap \mathcal B_n].
\end{equation}
Lemma~\ref{lem:3} allows to upper bound the first term on the right-hand side by $C_3\Delta_{4n}$, with $C_3$ depending on $\eta_1, \dots, \eta_8$. 
To bound the second term, observe  that for $\calE_n$ to occur, a translated version of $\calE_{\lfloor n/2\rfloor }$ by some  vector  $x=(\pm \lfloor n/2\rfloor,\pm \lfloor n/ 2\rfloor)$ must occur. 
These events are measurable with respect to the configuration inside $\Lambda_{3n/2}$, hence they  are independent of $\mathcal B_n$. By the union bound we obtain
\begin{equation}
  \label{eq:13}
  \mathbf P[\calE_n\cap \mathcal B_n]
  \le 4 \delta_{\lfloor n/2\rfloor}\mathbf P[\mathcal B_n]\overset{\eqref{eq:8}}{\le}\frac{c_1}2 \delta_{\lfloor n/2\rfloor}.
\end{equation}
Plugging \eqref{eq:13} and the bound given by Lemma~\ref{lem:3} into \eqref{eq:12}, we obtain
\begin{equation}\label{eq:22}
	\delta_n \le C_3\Delta_{4n}+\frac{c_1}2\delta_{\lfloor  n/2\rfloor}.
\end{equation}
Set now $r_n:=\delta_n/\Delta_{4n}$. 
Dividing the above by $\Delta_{4n}$ and using Lemma~\ref{lem:1} to bound $\Delta_{4\lfloor n/2\rfloor} /\Delta_{4n}$ by $1/c_1$, we find that 
\begin{align*}
	r_n\le C_3  + \tfrac12 r_{\lfloor  n/2\rfloor}
\end{align*} 
for every $n\in[\eta^{-1},\xi(p')]$.
An immediate induction shows that $r_n\le 2C_3+C_4$ for every $n\le \xi(p')$, 
where $C_4:=\max\{r_k \: : \: k\le \eta_1^{-1}, \,p'\ge p\ge p_c\}$. This concludes the proof. 
\end{proof}

\section{Derivation of Theorem~\ref{thm:1} from Theorem~\ref{thm:stability}}\label{sec:thm1}

As mentioned in the introduction, this step is classical and we simply sketch previously available arguments. 

\begin{proof}[Theorem~\ref{thm:1}]
Fix $p > p_c$. The constants $c_i,C_i$ below are independent of $p$ or of any spatial variables $n,N$. 

For the upper bound, the inclusion of events and Theorem~\ref{thm:stability} imply that 
\begin{equation}
	\theta(p)\le \theta_{\xi(p)}(p)\le C_0\theta_{\xi(p)}(p_c).
\end{equation}

We turn to the lower bound and start by proving the existence of $c_1>0$ such that
\begin{equation}\label{eq:conn}
	\mathrm P_p[\La_{\xi(p)} \text{ is connected to } \infty] \geq c_1.
\end{equation} 

\begin{figure}
\begin{center}
\includegraphics[height = 6cm]{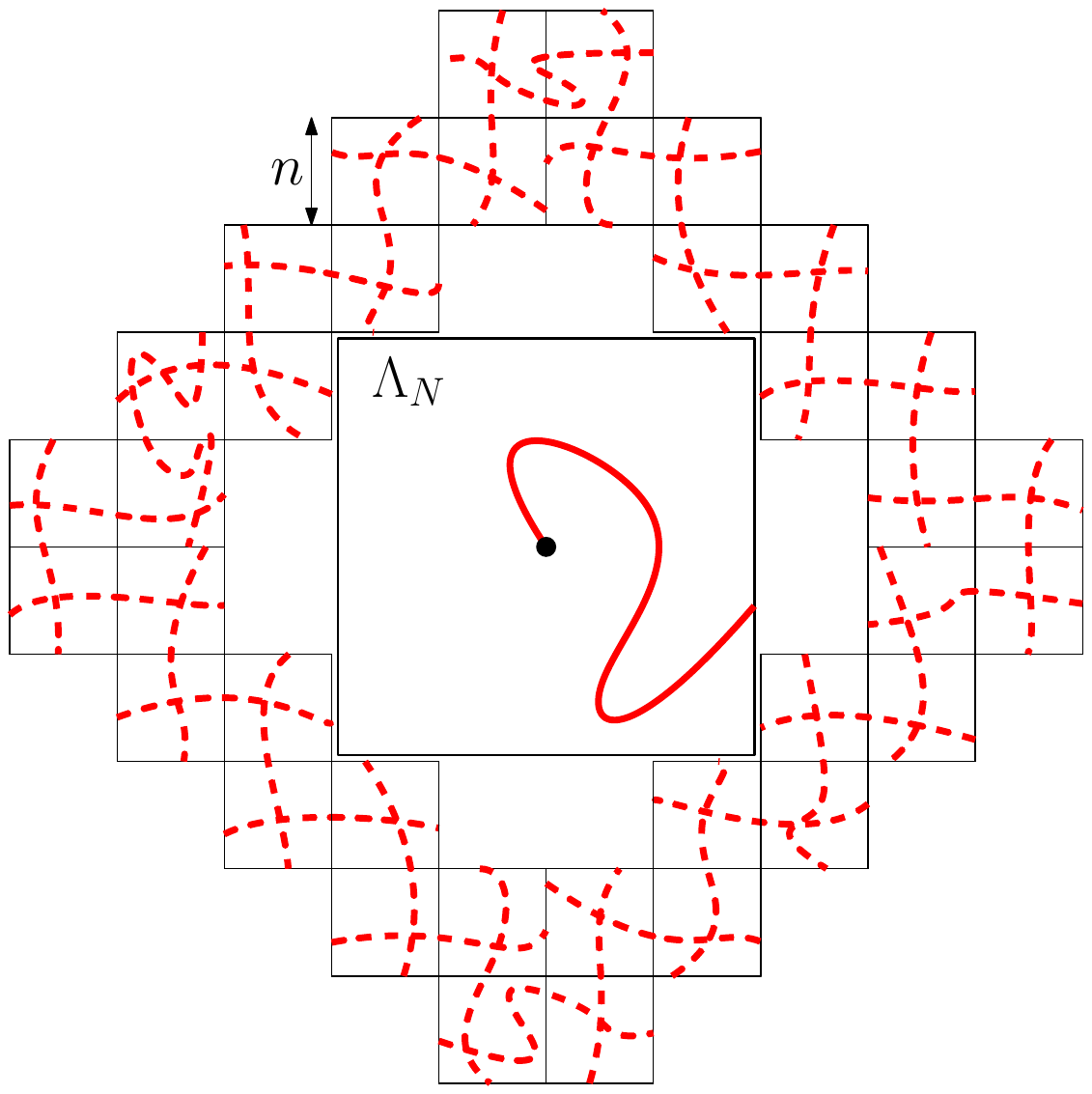}\hspace{.04\textwidth}
\includegraphics[height = 6cm]{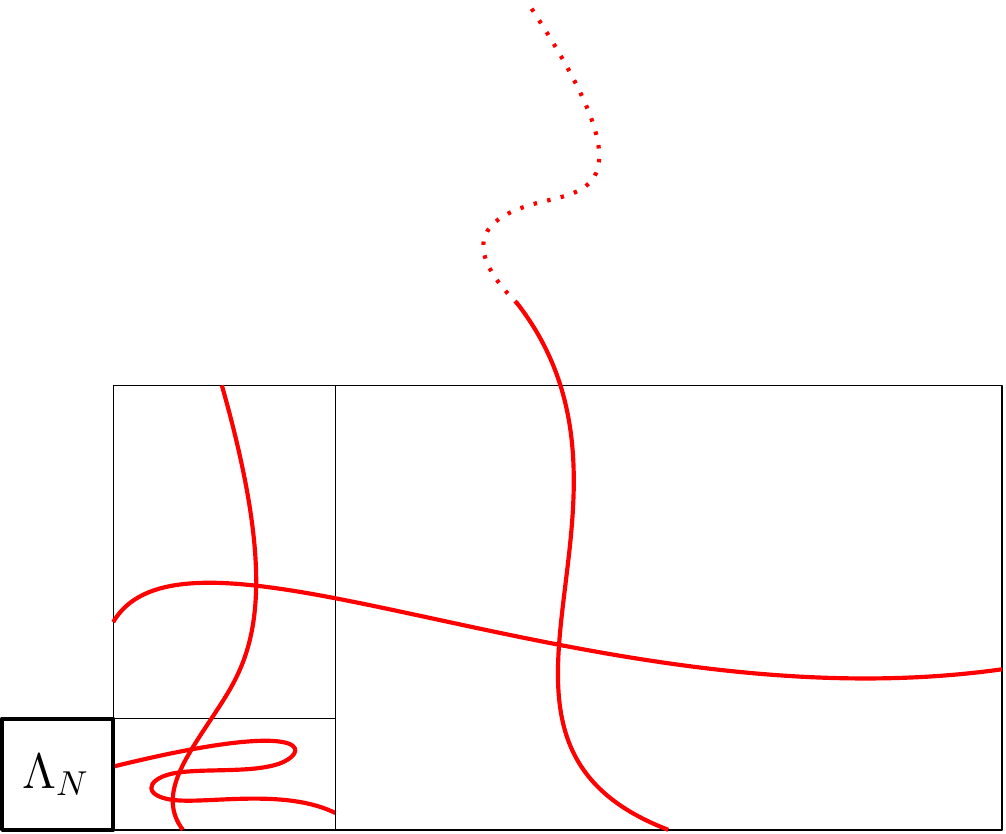}
\caption{{\em Left:} combining dual crossings in the long directions of $16\lceil N/n\rceil$ $2n\times n$ rectangles produces a dual circuit around $\La_N$. If in addition $0$ is connected to $\partial \La_N$ by a primal path, the configuration contributes to $\theta_N(p) - \theta(p)$. \newline 
{\em Right:} combining primal crossings of rectangles $2^{k+1}\xi(p)\times 2^{k}\xi(p)$ for $k \geq 1$ 
produces an infinite cluster intersecting $\La_{\xi(p)}$. Due to \eqref{eq:metd} and the FKG inequality, 
these crossings occur simultaneously with uniformly positive probability.}
\label{fig:xi}
\end{center}
\end{figure}

To do this, observe that combining dual crossings using the FKG inequality as in the left diagram of Fig.~\ref{fig:xi} 
implies that for every $N\ge n\ge 1$,
\begin{align*}
    \theta_N(p) - \theta(p) 
    &\geq \theta_N(p)\mathrm P_{1-p}[\mathcal C([0,2n]\times[0,n])]^{16\lceil N/n\rceil}\\
    &\geq \theta(p) \mathrm P_{1-p}[\mathcal C([0,n]\times[0,2n])]^{C_2\, N/n},
\end{align*}
where the second inequality is due to the Russo-Seymour-Welsh theorem \eqref{eq:RSW} and the obvious bound $\theta_N(p) \geq \theta(p)$.  
Letting $N$ go to infinity, using the definition of $\xi(p)$ and the fact that $\theta(p) > 0$ implies that 
\begin{align}\label{eq:metd}
\mathrm P_{1-p}[\mathcal C([0,n]\times[0,2n])]\le \exp[-c_3 n/\xi(p)].
\end{align}
Using duality\footnote{We use the self-duality of the square lattice but the arguments can be adapted to fit any lattice.} we deduce that 
\[
\mathrm P_{p}[\mathcal C([0,2n-1]\times[0,n])]\ge 1-\exp[-c_3 n/\xi(p)].
\]
Combining crossing events in rectangles of size $2^k\xi(p)$ as in the right diagram of Fig.~\ref{fig:xi} implies \eqref{eq:conn}.

The quasi-multiplicativity of the one arm (Lemma~\ref{lem:quasi}), \eqref{eq:conn}, and the monotonicity in $p$ imply that 
\begin{equation}
C\theta(p) \ge \theta_{\xi(p)}(p) \mathrm P_p[\La_{\xi(p)} \text{ is connected to } \infty] \ge c_1  \theta_{\xi(p)}(p_c),
\end{equation}
which is the desired lower bound. 
\end{proof}

\appendix

\section{Appendix}\label{sec:lem}

In this appendix we derive some standard results necessary for our proof.
The only prior knowledge that is required is the classical Russo-Seymour-Welsh (RSW) theorem \cite{Rus78,SeyWel78,Gri99,KST21},
which states that for every $\rho>0$, there exists $C=C(\rho)>0$ such that 
\begin{equation}\label{eq:RSW}
	\mathrm P_p[\calC([0,n]\times[0,\rho m])]\ge \mathrm P_p[\calC([0,n]\times[0,m])]^C \qquad \text{ for all $m,n \geq 1$ and every $p$}.
\end{equation}
A first consequence of the above is the box crossing property below the correlation length. 

\begin{proposition}[Box crossing property] \label{prop:crossing} 
    For every $\rho>0$, there exists $c(\rho)>0$ such that for every $p$, every $n\le \xi(p)$ and every $m\in[\tfrac1\rho n,\rho n]$,
    \begin{align}\label{eq:BXP}
    	c(\rho)\le \mathrm P_p[\calC([0,n]\times[0,m])]\le 1-c(\rho).
    \end{align}
\end{proposition}

\begin{proof}
We prove the lower bound for $p \le p_c$. A similar argument applied to the dual model proves the upper bound for $p \ge p_c$. 
The upper bound for $p < p_c$ and the lower bound for $p > p_c$ follow by monotonicity.

Fix $p \le p_c$ and $1\le n\le \xi(p)$. 
Consider the site percolation model $\eta$ on $n\mathbb Z^2$ for which $\eta_x=1$ if $\Lambda_n(x)$ is connected to $\Lambda_{2n}(x)$ in $\omega_p$, and otherwise $\eta_x=0$. For any multiple $N$ of $n$, 
if $0$ is connected to $\partial\Lambda_{N}$ in $\omega_p$, 
there necessarily exists a self-avoiding path of vertices $x_0=0,\dots,x_{N/n}$ in $n\mathbb Z^2$ with $\|x_i-x_{i+1}\|_2=n$ 
and $\eta_{x_i}=1$ for every $0\le i<N/n$. 
For any such possible path, one may extract a subfamily of $\lfloor N/(n|\Lambda_4|)\rfloor$ vertices $x_{i(j)}$ 
that are at a $\ell^\infty$-distance at least $4n$ of each other. 
The $\eta_{x_{i(j)}}$ variables are then independent Bernoulli random variables of parameter 
$q:=\mathrm P_p[\Lambda_n\text{ is connected to }\partial\Lambda_{2n}]$. 
Doing a union bound on possible paths, we deduce that 
\begin{equation}
 \theta_N(p)\le 4^{N/n}q^{\lfloor N/(n|\Lambda_4|)\rfloor}.
\end{equation}
Letting $N$ tend to infinity and using the definition of $\xi(p)$ implies that 
\begin{equation}\label{eq:a1}
	q\ge (4e)^{-|\Lambda_4|}=:c_0.
\end{equation}

Now, for $\Lambda_n$ to be connected to $\partial\Lambda_{2n}$, at least one out of four rectangles of size $4n\times n$ must contain a crossing between its two larger sides. We deduce that 
\begin{equation}\label{eq:a2}
\mathrm P_p[\calC([0,n]\times[0,4n])]\ge c_0/4.
\end{equation}
The result then follows from the RSW theorem \eqref{eq:RSW}.
\end{proof}

The second application is related to the quantity $\theta_{n,N}(p)$, 
which we recall denotes the probability that $\Lambda_n$ is connected to $\partial\Lambda_N$ in $\omega_p$.

\begin{lemma}[Quasi-multiplicaticity of the one arm probability]\label{lem:quasi}
There exists $C>0$ such that for every $p$ and every $n\le k\le \xi(p)$ and $N\ge k$,
\begin{equation}\label{eq:}
\theta_{n,N}(p)\le \theta_{n,k}(p)\theta_{k,N}(p)\le C\theta_{n,N}(p).
\end{equation} 
\end{lemma}

\begin{proof}
The left inequality is obvious by inclusion of events. For the right one, let $\calE_k$ be the event that in $\omega_p$, $\Lambda_{k/2}$ is connected to $\partial\Lambda_{2k}$ and there exist circuits in $\Lambda_{k}$ surrounding $\Lambda_{k/2}$ and in $\Lambda_{2k}$ surrounding $\Lambda_{k}$. The box crossing property below the correlation length \eqref{eq:BXP} and the FKG inequality imply a uniform lower bound for the  probability of $\calE_k$.
Finally, if $\calE_k$ occurs and if there exist connections in $\omega_p$ from $\partial\Lambda_n$ to $\partial\Lambda_k$, and  from $\partial\Lambda_k$ to $\partial\Lambda_N$, these combine to produce a path in $\omega_p$ from $\partial\Lambda_n$ to $\partial\Lambda_N$.
The FKG inequality produces the desired conclusion. 
\end{proof}

Finally, we prove the upper bound on the probability of the mixed three arm event used in the proof of Lemma~\ref{lem:2}
			
\begin{proof}[Claim~\ref{claim:1}]
	Fix $x$, $p$, $p'$, $r$, $R$ and $n$ as in the statement. It suffices to treat the case where $R \geq 4r$. 
	First observe that, as a consequence of \eqref{eq:BXP}, 
	the probabilities that $\La_{r}(x)$ is connected to $\partial \La_R(x)$ in $\omega_p$ and in $\omega_{p'}^\star$, respectively, 
	are both bounded by $(r/R)^\delta$ for some universal $\delta > 0$. 
	
	\begin{figure}
	\begin{center}
	\includegraphics[width = .8\textwidth]{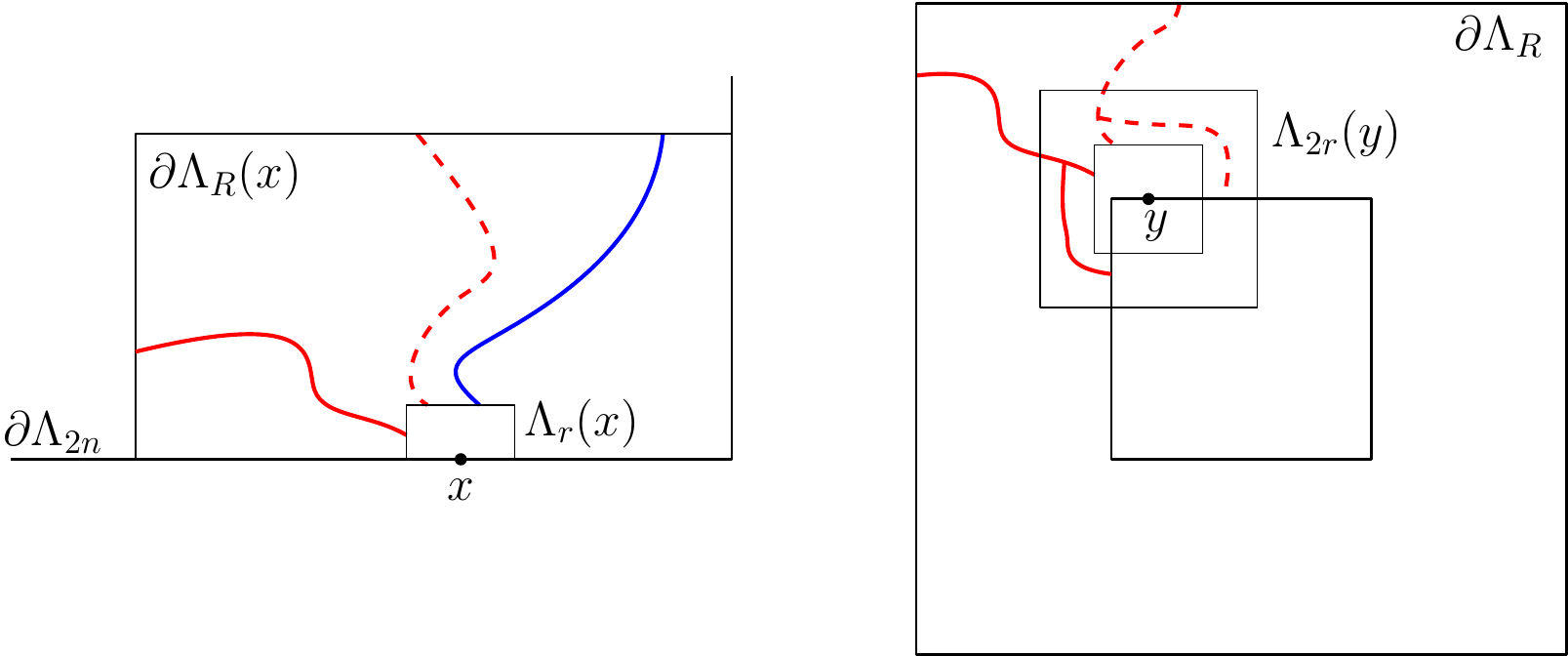}
	\caption{{\em Left:} a mixed three arm event between $\La_r(x)$ and $\La_R(x)$ for some $x \in \partial \La_{2n}$. 
	The red paths are contained in $\omega_p$ and $\omega^\star_p$, the blue one is contained in $\omega_{p'}$. \newline
	{\em Right:} The translation (and potential rotation) of the configuration $\omega_p$ produces primal and dual connection between $\La_{r}(y)$ and $\partial \La_{R}$ contained in $\La_R \setminus \La_{R/2}$. Using the box-crossing property to add the two arcs in $\La_{2r}(y) \setminus \La_{r}(y)$, 
	we create the event in the right-hand side of \eqref{eq:2arms_boundary} at a constant cost.}
	\label{fig:2arms}
	\end{center}
	\end{figure}

	Applying  Reimer's inequality\footnote{The inequality applies here since $(\omega_p,\omega_{p'})$
	may be generated by two independent Bernoulli percolations, one with parameter $p$ and one with parameter $p'-p$.} 
	\cite{Rei00} to the joint law of $(\omega_p,\omega_{p'})$, we find that 
	\begin{align}\label{eq:Rim}
    	&\mathbf P [(\omega_p,\omega_{p'}) \in \calA_{\rm mix}(x; r,R)] \\
    	& \qquad \leq (r/R)^\delta\, \mathrm P_{p}[\La_{r}(x) \text{ connected to }\partial \La_R(x) \text{ by primal and dual paths in $\La_{2n}$}]\nonumber\\
    	& \qquad +(r/R)^\delta \, \mathrm P_{p'}[\La_{r}(x) \text{ connected to }\partial \La_R(x) \text{ by primal and dual paths in $\La_{2n}$}].\nonumber
	\end{align}
	We will focus on the first term. Applying the box-crossing property for both the primal and dual models of $\mathrm P_p$ (see Fig.~\ref{fig:2arms}), 
	we conclude that, for every $x \in \partial \La_{2n}$ and $y \in \partial \La_{R/2}$,
	\begin{align}\label{eq:2arms_boundary}
    	&\mathrm P_{p}[\La_{r}(x) \text{ connected to }\partial \La_R(x) \text{ by primal and dual paths in $\La_{2n}$}]\\
    	&\quad  \leq C \, \mathrm P_{p}[\La_{2r}(y) \cap \partial \La_{R/2} \text{ connected to }\partial \La_{R} \text{ by primal and dual paths in $\La_{R} \setminus \La_{R/2}$}].\nonumber
	\end{align}
	Summing the above over all $y \in \partial \La_{R/2}$, we find 
	\begin{align}\label{eq:2arms_boundary2}
    	& |\partial \La_{R/2}|\, \mathrm P_{p}[\La_{r}(x) \text{ connected to }\partial \La_R(x) \text{ by primal and dual paths in $\La_{2n}$}]\\
    	& \quad \leq 8r C \, \mathrm E_{p}[\text{number of clusters in $\La_{R} \setminus \La_{R/2}$ intersecting both $\partial \La_{R/2}$ and $\partial \La_{R}$}],\nonumber		
	\end{align}
	where $ \mathrm E_{p}$ denotes the expectation with respect to $\mathrm P_{p}$. 
	Indeed, for $\La_{2r}(y) \cap \partial \La_{R/2}$ to be connected to $\partial\La_{R}$ by both primal and dual paths, 
	$y$ needs to be within a distance $2r$ of the endpoint of a primal/dual interface starting on $\partial \La_{R}$ and ending on $\partial\La_{R/2}$. 
	For each such interface, there exist at most $4r$ points $y \in \partial \La_{R/2}$ as above.	
	Moreover, the number of interfaces is twice the number of clusters in the right-hand side of \eqref{eq:2arms_boundary2}.
	
	It is a simple consequence of the box-crossing property \eqref{eq:BXP} that the number of clusters in~\eqref{eq:2arms_boundary2} is dominated by a geometric variable with intensity independent of $R$. Thus, the right-hand side of~\eqref{eq:2arms_boundary2} is uniformly bounded. The same bound holds for $\mathrm P_{p'}$, and inserting these two bounds in~\eqref{eq:Rim} leads to the desired inequality. 
\end{proof}

\newcommand{\etalchar}[1]{$^{#1}$}
\providecommand{\bysame}{\leavevmode\hbox to3em{\hrulefill}\thinspace}
\providecommand{\MR}{\relax\ifhmode\unskip\space\fi MR }
\providecommand{\MRhref}[2]{%
  \href{http://www.ams.org/mathscinet-getitem?mr=#1}{#2}
}
\providecommand{\href}[2]{#2}

\end{document}